\documentclass[preprint, 12pt]{elsarticle}
\usepackage{amssymb}

\usepackage{amsfonts,amsmath,amsthm}
\usepackage{amssymb}
\usepackage[utf8]{inputenc}
\usepackage[english]{babel}
\usepackage{graphicx}
\usepackage{mathtools}
\usepackage{epstopdf}
\usepackage{tikz}
\usetikzlibrary{shapes,arrows,backgrounds}
\usepackage{subcaption}
\usepackage{array}
\usepackage{url}
\usepackage{booktabs}
\usepackage{lineno,hyperref}

\journal{Journal of Mathematical Analysis and Applications}

\newtheorem{theorem}{Theorem}[section]
\newtheorem*{theorem-non}{Theorem}
\newtheorem{lemma}[theorem]{Lemma}

\newtheorem*{proposition-non}{Proposition}
\newtheorem{prop}[theorem]{Proposition}

\newtheorem{coro}[theorem]{Corollary}

\newtheorem{definition}[theorem]{Definition}
\newtheorem{example}[theorem]{Example}
\newtheorem{conjecture}[theorem]{Conjecture}
\newtheorem{obs}[theorem]{Observation}
\newtheorem*{obs-non}{Observation}
\newtheorem*{conjecture-non}{Conjecture}
\theoremstyle{remark}
\newtheorem*{question-non}{Question}
\newtheorem{question}[theorem]{Question}

\newtheorem{claim}[theorem]{Claim}
\newtheorem{remark}[theorem]{Remark}

\newcommand{\minus}{\scalebox{0.75}[1.0]{$-$}}
\def\cl{\overline}
\def\LH{\stackrel{*}=}
\def\sq22{\sqrt{2-2\cos(t)}}
\def\K{\mathcal{K}}

\def\interior{\operatorname{int}}
\def\conv{\operatorname{conv}}
\def\Bis{\operatorname{Bis}}

\def\Id{\operatorname{Id}}

\def\MUP{Mazur-Ulam Property}
\def\R{\mathbb R}

\def\B{\mathcal B}
\def\el2{\mathcal D}
\def\e{\varepsilon}
\def\norma{\|\cdot\|}
\def\nX{\|\cdot\|_X}
\def\XnX{(X,\|\cdot\|_X)}
\def\RnX{(\R^n,\|\cdot\|'_X)}

\def\nY{\|\cdot\|_Y}
\def\YnY{(Y,\|\cdot\|_Y)}
\def\RnY{(\R^n,\|\cdot\|'_Y)}

\def\tauSXY{\tau:S_X\to S_Y}
\def\tauSXX{\tau:S_X\to S_X}
\def\phiSXX{\phi:S_X\to S_X}

\def\tauXY{\widetilde{\tau}:X\to Y}
\def\tauXYn{\widetilde{\tau}:\XnX\to\YnY}
\def\wtau{\widetilde\tau}

\begin{document}

\begin{frontmatter}

\title{A reflection on Tingley's problem and some applications}

\author
{Javier~Cabello~Sánchez } 
\address{Departamento de Matem\'{a}ticas, Universidad de Extremadura, 
Avda. de Elvas s/n, 06006 Badajoz. Spain. 
{\em Dedicated to the memory of Professor Carlos Benítez.}}
\ead{coco@unex.es}

\begin{abstract}
In this paper we show how some metric properties of the unit 
sphere of a normed space can help to approach a solution to Tingley's problem. 
In our main result we show that if an 
onto isometry between the spheres of strictly convex spaces is the identity when 
restricted to some relative open subset, then it is the identity. This implies that an onto 
isometry between the unit spheres of strictly convex finite dimensional spaces 
is linear if and only it is linear on a relative open set. We prove the same for arbitrary 
two-dimensional spaces and obtain that every two-dimensional, non strictly convex, 
normed space has the Mazur-Ulam Property. 

We also include some other less general, yet interesting, results, 
along with a generalisation of curvature to normed spaces. 
\end{abstract}


\end{frontmatter}

% \maketitle

\section{Introduction}

Since Tingley's seminal paper~\cite{tingley}, a lot of work has been done trying to answer this: 

\begin{question}\label{qt}
Let $(X,\nX), (Y,\nY)$ be normed spaces and $\tauSXY$ a surjective isometry 
between their unit spheres. Is $\tau$ the restriction of a linear isometry $\tauXY$?
\end{question}

This Question is widely known as Tingley's problem. 
The main result in Tingley's paper~\cite{tingley} is what we 
will call {\em Tingley's Theorem} throughout the paper: 

\begin{theorem}[Tingley, \cite{tingley}]\label{Tingley}
Suppose that $S$ and $S'$ are the unit spheres of finite dimensional Banach spaces. 
If $f:S\to S'$ is an onto isometry, then $f(-x)=-f(x)$ for all $x$ in $S$. 
\end{theorem}

Of course, this is not the first question about the linearity of isometries. 
Namely, since Mazur-Ulam Theorem, see~\cite{MazurUlam}, we are aware of the fact that 
every surjective isometry $\tauXYn$ is affine. The relation between the Theorem and the 
Question has led to state (see, e.g.,\cite{chengdong}) that $\XnX$ has the Mazur-Ulam 
Property when the answer to Question~\ref{qt} is affirmative for every $\YnY$. 
As a consequence of Mazur-Ulam Theorem, Question~\ref{qt} may be replaced by the following:

$\bullet$ 
Is $\tau$ the restriction of an isometry $\tauXY$ such that $\wtau(0)=0$?

The next natural step could have been the question about surjective isometries 
between unit balls, but Mankiewicz, in~\cite{Mankiewicz}, showed that in the 
Mazur-Ulam Theorem the surjective isometry does not need to be defined on the 
whole space $X$ or even on the unit ball: if we consider two closed convex 
bodies $F_X\subset X$ and $F_Y\subset Y$, every onto isometry $\tau:F_X\to F_Y$ 
is affine --it is a little more general, actually. 
So, Tingley's problem can be then restated as: 

$\bullet$ 
Is $\tau$ the restriction of an isometry $\wtau:B_X\to B_Y$ such that $\wtau(0)=0$?

It seems that, if every space has the Mazur-Ulam Property, the last question 
of this kind will be, in the spirit of Mankiewicz's result:

$\bullet$ 
If $F_X$ and $F_Y$ are convex bodies and $\tau:\partial F_X\to \partial F_Y$ 
is an onto isometry between their boundaries, does $\tau$ extend linearly? 

This problem has been dealt in several ways and lots of positive answers 
have been found, see, 
e.g.,~\cite{BCFP18, CP18, ding2002, ding2013, 
fangwang2006, PeraltaNG, FPeralta2017, FPeralta2018CBH, FPeralta2018Low, 
FPeralta2018vN, JVMCPR19, KadetsMiguel, li2016, 
Mori18, MoriOzawa19, Peralta2018Acta, Peralta2019RevMC, 
PeraTa2019, CLTanLiu, tanaka2014, Tanaka2016matr, 
Tanaka2017vN, wanghuang} --it is really 
impressive the development of machinery and technics that this problem has led to. 

Anyway, all usual approaches share a common procedure: take some more o less 
concrete normed space $(X,\nX)$ and its unit sphere $S_X$, suppose that for 
some $(Y,\nY)$ --with or without further assumptions on $\YnY$-- there is some 
onto isometry $\tauSXY$, analyse some properties of the involved norms and show somehow 
that the homogeneous extension of $\tau$ is either isometric or linear. 
This is enough because, in this setting, linear implies isometric and vice versa. 
When we say {\em the homogeneous extension} we refer to 
$\wtau:X\to Y$ defined as $\wtau(\lambda x)=\lambda\tau(x)$ for every 
$\lambda\geq 0, x\in S_X$. 

\vspace{4mm}

Our approach will follow a different way: if $\tau$ is the restriction 
of a linear isometry, in particular $\tau$ {\em must be linear}. We mean 
that, whenever $x,x'\in S_X$ and $\lambda, \lambda'\in \R$ are such that 
$\lambda x+\lambda'x'\in S_X$, the point $\lambda\tau(x)+\lambda'\tau(x')$ 
must belong to $S_Y$ and the equality 
$$\tau(\lambda x+\lambda'x')=\lambda\tau(x)+\lambda'\tau(x')$$
must hold. So, taking coordinates with respect to well chosen bases, we will have 
that if $\tau$ is the restriction of some linear isometry then its representation 
in coordinates must be the identity. Of course, the identity of a sphere is the 
restriction of a linear isometry, so this is a necessary and sufficient 
condition for $\tau$ to be the restriction of a linear isometry. 

As our immediate goal is not to detail thoroughly our method, we will try to 
explain it in its simplest form. Consider some two-dimensional normed space 
$(X,\nX)$ and suppose that there is a surjective isometry 
$\tauSXY$ for some $(Y,\nY)$ --it must be two-dimensional, too. 
We may take a basis $\B_X=\{x_1,x_2\}\subset S_X$ and identify {\em linearly} isometrically 
$(X,\nX)$ with $(\R^2,\nX')$ and $\B_X$ with the usual basis $\B_2=\{e_1,e_2\}$ 
in $\R^2$: 
$$\phi_X(\lambda_1x_1+\lambda_2x_2)=(\lambda_1,\lambda_2), \quad 
\|(\alpha_1,\alpha_2)\|'_X=\|\alpha_1x_1+\alpha_2x_2\|_X.$$
If we take, further, $y_1=\tau(x_1), y_2=\tau(x_2)$ and $\B_Y=\{y_1,y_2\}\subset S_Y$, 
then $\B_Y$ is a basis of $Y$ by Tingley's Theorem. Identifying the same 
way $Y$ with $\R^2$ and $\B_Y$ with $\B_2$ via $\phi_Y:Y\to\R^2$ we may consider the 
map $\tau':S_{\nX'}\subset\R^2\to S_{\nY'}\subset\R^2$ defined as 
$\tau'(\lambda_1,\lambda_2)=(\mu_1,\mu_2)$ when  
$\tau(\lambda_1x_1+\lambda_2x_2)=\mu_1y_1+\mu_2y_2, $ 
i.e., $\tau'=\phi_Y\circ\tau\circ\phi^{-1}_X. $
As $\tau'(1,0)=(1,0)$ and $\tau'(0,1)=(0,1)$, the only way $\tau'$ can be linearly extended 
is being the identity so, since $\phi_X$ and $\phi_Y$ are linear, we have two options: 

\begin{itemize}
\item If $\tau'$ is the identity then $\tauSXY$ is the restriction of a linear isometry. 
\item If $\tau'$ is not the identity then there is no linear 
application whose restriction agrees with $\tau$. 
\end{itemize}

So, for $\tauXY$ to be linear it is necessary $S_{\nX'}=S_{\nY'}$. 
We will see that this is also sufficient in two-dimensional spaces, so 
the planar Tingley's problem could be stated as follows: 

\begin{question}\label{q2d}
Suppose $\nX$ and $\nY$ are two norms defined on $\R^2$ such that 
$\|(1,0)\|=\|(0,1)\|=1$ for both norms and there is an isometry $\tauSXY$ such that 
$\tau(1,0)=(1,0), \tau(0,1)=(0,1)$. Does this imply $S_X=S_Y$?
\end{question}

\subsection{Notations}

Throughout this paper, $X$ and $Y$ will be normed spaces. When we deal with 
more than one space, we will write $\nX, \nY$ and so on for the norms unless 
we are referring, on purpose, to equalities or relations that hold for all 
the involved norms, as in Question~\ref{q2d}. 

We will denote by $\B_n=\{e_1,\ldots,e_n\}$ the usual basis of $\R^n$, in particular 
every appearance of $e_i$ will refer to the $i$-th vector of $\B_n$. 

We have avoided the use of open intervals or segments, so that the notation 
$(a,b)$ will always refer to a two-dimensional vector. For closed intervals 
or segments, we will write $[x,x']$, i.e, $[x,x']=\{\lambda x+(1-\lambda x'):
0\leq\lambda\leq 1\}$ no matter whether $x$ and $x'$ are scalars or vectors. 

Given $x,x'\in X$ we will denote as $\Bis(x,x')$ the bisector of the segment 
$[x,x']$, i.e., $\Bis(x,x')=\{z\in X:\|x-z\|_X=\|x'-z\|_X\}$. As we will deal 
frequently with bisectors of symmetric segments of the form $\Bis(x,-x)$, we 
will refer to them as symmetric bisectors and will omit the $-x$ in the notation, 
so $\Bis(x)$ will be the {\em symmetric bisector of} $x$ and must be understood as $\Bis(-x,x)$. 
Please observe that $z\in \Bis(x)$ if and only if $x$ and $z$ are isosceles orthogonal. 

\begin{definition}
Given a segment $[x,x']$ in the sphere of some two-dimensional normed space 
$\XnX$, we say that $[x,x']$ is maximal when it is not strictly contained in 
another segment. 
\end{definition}

\begin{definition}
For $x\in S_X$, the star of $x$ is $\{x'\in S_X:[x,x']\subset S_X\}$. 
\end{definition}

As we will focus primarily on metric concepts, the following subset will 
play the usual role of the star. 
The definition of the star is here just for, say, compatibility purposes. 

\begin{definition}\label{eld2}
Given $x\in S_X$, we will denote by $\el2(x)$ the set $\{x'\in S_X:\|x-x'\|_X=2\}$. 
\end{definition}

\begin{remark}\label{el2star}
By Corollary 5 in~\cite{tingley}, $\el2(x)$ agrees with the star of $-x$, 
so $\el2(x)=\{x'\in S_X:[-x,-x']\subset S_X\}$. 
\end{remark}

\subsection{Plan of the paper}

Apart from this Introduction, the paper is divided into three sections. 

The first one is devoted to some quite elementary, general facts 
about spheres in normed spaces. These facts will be useful in the 
remaining two sections. 

Section 3 is the central one, we have split the proof of the main results into 
several parts. Some of these intermediate results are interesting on their own, 
and also some proofs reveal the main ideas in this paper far better than the 
main results' proofs. 

Finally, in Section 4 we expose some results than can be seen as consequences of 
the ideas more than consequences of the results in Section 3. It includes 
a subsection where we define a kind of generalisation of the usual curvature of 
planar curves. 

\section{The general results}

Given a finite dimensional $\XnX$ and a basis $\B_X=\{x_1,\ldots,x_n\}\subset X$, we 
say that $(\XnX,\B_X)$ is identified with $\RnX$ if 
$$\|(\lambda_1,\ldots,\lambda_n)\|'_X=\|\lambda_1x_1+\cdots\lambda_nx_n\|_X$$ 
for every $(\lambda_1,\ldots,\lambda_n)\in \R^n$, i.e., if the linear application 
$\phi_X:\XnX\to\RnX$ given by $\phi_X(x_i)=e_i, i=1,\ldots,n$ is an isometry. 

\begin{lemma}\label{fundamental}
Let $\XnX, \YnY$ be finite dimensional normed spaces, $\tauSXY$ an 
application (respectively, isometry), and suppose that $\B_X=\{x_1,\ldots,x_n\}\subset S_X$ 
and $\B_Y=\{\tau(x_1),\ldots,\tau(x_n)\}\subset S_Y$ are bases of $X$ and $Y$ 
respectively. Suppose, moreover, that $(\XnX,\B_X)$ and $(\YnY,\B_Y)$ are identified 
with $\RnX$ and $\RnY$ via $\phi_X$ and $\phi_Y$. Then, $\tau$ is the restriction 
of a linear application (resp, isometry) $\tauXY$ if and only if 
$\tau'=\phi_Y\circ\tau\circ\phi_X^{-1}:S_{\nX'}\to S_{\nY'}$ is the identity. 
\end{lemma}

\begin{proof}
Since $\phi_X$ and $\phi_Y$ are onto linear isometries and $\tau'(e_i)=e_i$ for every $i$, the 
results follows. 
\end{proof}

\begin{definition}
We say that a norm $\nX$ defined on $\R^n$ is normalized if $\|e_i\|_X=1$ for 
every $e_i\in \B_n$. 
\end{definition}

\begin{theorem}\label{twod}
Let $\nX, \nY$ be normalized norms defined on $\R^2$, $\tauSXY$ an 
isometry such that $\tau(1,0)=(1,0)$ and $\tau(0,1)=(0,1)$. Then $\tau$ is the 
restriction of a linear isometry if and only if $S_X=S_Y$.
\end{theorem}

\begin{proof}
Let $\tau$ be as in the statement. By Lemma~\ref{fundamental}, if $S_X$ and 
$S_Y$ are different, then $\tau$ is not linear, so the ``only if" part is done. 
What we need to show in order to prove the other implication is that $S_Y=S_X$ 
implies that the only isometry $\tauSXY$ is the identity. 

Take some $z, \tau(z)\in S_X$. As $\tau$ preserves distances and 
$\tau(e_1)=e_1, \tau(e_2)=e_2$, we have $\|z\pm e_1\|=\|\tau(z)\pm e_1\|, $ and 
$ \|z\pm e_2\|=\|\tau(z)\pm e_2\|$. 
The Monotonicity lemma (see, e.g., \cite{SurveyMSW1},~Proposition~31) 
implies that for any normed plane $E$ and any basis 
$\{u_1,u_2\}\subset E$, the four distances $\|v\pm u_1\|_E$, $\|v\pm u_2\|_E$ 
determine $v$ when $v, u_1, u_2\in S_E$, so we have $z=\tau(z)$. 
\end{proof}

\begin{remark}
This result is very similar to~\cite{tanakaR2},~Corollary~2.12.
\end{remark}

\begin{conjecture}\label{conjrn}
Theorem~\ref{twod} is true for any couple of normalized norms defined on $\R^n$. 
\end{conjecture}

Of course this conjecture is a particular case of Tingley's Problem, we explicited 
it here just because it seems much easier to answer and it could be helpful. 
See Remark~\ref{remrm} for a little further explanation.

\begin{remark}
The proof of Theorem~\ref{twod} will not adapt to this Conjecture, here we explicit a simple example 
of a three-dimensional space with a basis that does not determine the points in the 
sense used above. Take $(\R^3,\norma_\infty)$ and the basis 
$\{v_1,v_2,v_3\}=\{(1,1,1),(1,1,0.9),(1,0.9,1)\}$. We need to point out that all the 
coordinates in this Remark refer to the usual basis. It is clear that, for 
$y_1=(1,-1,0.1), y_2=(1,-1,-0.1)$, we have 
\begin{itemize}
\item $\|v_i+y_j\|=2$ for every $i, j$. 
\item $\|v_i-y_j\|=2$ if $i, j\in\{1,2\}$.
\item $\|v_3-y_j\|=1.9$ for $j\in\{1,2\}$.
\end{itemize}

On the other hand, it is not hard 
to see that the only isometry of $S^3_\infty$ that preserves $\{v_1,v_2,v_3\}$ 
is the identity. Indeed, suppose $\tau:S^3_\infty\to S^3_\infty$ is an onto 
isometry and $\tau(v_i)=v_i, i=1, 2, 3.$

Then, $\tau$ preserves $\el2(v_1)\cap\el2(v_2)\cap\el2(v_3)=\{-1\}\times[-1,1]\times[-1,1]$, 
so it preserves $\{1\}\times[-1,1]\times[-1,1]$, too. Now it is clear that 
$\tau$ also preserves 
$$(\el2(v_1)\cap\el2(v_2))\setminus(\el2(v_1)\cap\el2(v_2)\cap\el2(v_3))=
\interior([-1,1]\times\{-1\}\times[-1,1])$$
and also $[-1,1]\times[-1,1]\times\{-1\}$, so $\tau(C)=C$ when $C$ is any of the 
six faces of the unit sphere. It is clear that every $(a_1,a_2,a_3)$ belonging 
to the {\em ball} $B^3_\infty$ is determined by its distances to the 
six faces of $S^3_\infty$, so $\tau(a_1,a_2,a_3)=(a_1,a_2,a_3)$, 
for every $(a_1,a_2,a_3)\in S^3_\infty$.
\end{remark}

We can explicit another example, this one involves a strictly convex space. 
Consider $(\R^3,\norma_3)$ and take $x=\frac 1{\sqrt[3]{3}}(1,1,1)\in S_3^3$. Then, 
$$v_1=\frac 1{\sqrt[3]{6}}(1,1,-\sqrt[3]{4}), v_2=\frac 1{\sqrt[3]{6}}(1,-\sqrt[3]{4},1), 
v_3=\frac 1{\sqrt[3]{6}}(-\sqrt[3]{4},1,1)$$
form a basis such that $\|v_i-x\|_3=\|v_j+x\|_3$ for $i, j\in\{1,2,3\}$, so this 
basis does not distinguish $x$ and $-x$. Indeed it is easy to check that all these 
quantities equal $\left(\frac 43+2\sqrt[3]{2}\right)^{1/3}$. What we have done is to 
choose the simplest $x\in S^3_3$ whose symmetric bisector is not planar and the 
simplest basis contained in its symmetric bisector: $v_1, v_2, v_3\in \Bis(x)$. 
Now, the following seems pretty natural:

\begin{question}
Let $\XnX$ be a three-dimensional normed space. Can we choose 
$x, v_1, v_2, v_3\in S_X$ such that $\{v_1, v_2, v_3\}$ 
is a basis and $v_1, v_2, v_3\in \Bis(x)$ whenever $X$ is not Euclidean?
\end{question}

\begin{remark}
It is clear from~\cite{BAustAdrian}, Theorem 3.2, that in every not Euclidean 
three-dimensional space $\XnX$, for each $\lambda\in(0,1)\cup(1,\infty)$ 
there exist $x, v_1, v_2, v_3\in S_X$ such that $v_1, v_2, \lambda v_3\in \Bis(x)$ 
and $\{v_1, v_2, v_3\}$ is a basis. 
\end{remark}

\section{The main results}

We have tried to explicit every useful property, so we have split the proof 
of the main results into several intermediate steps.

\begin{remark}\label{x-x}
Consider an onto isometry $\tauSXY$ between the spheres of a pair of finite dimensional 
normed spaces. By Tingley's Theorem, for every $x\in S_X$ we have $\tau(-x)=-\tau(x)$, so 
\small
$$\|\tau(x_1)+\tau(x_2)\|_Y\!=\!\|\tau(x_1)\minus (\minus \tau(x_2))\|_Y\!=\!\|\tau(x_1)\minus \tau(\minus x_2)\|_Y\!=\!
\|x_1\minus (\minus x_2)\|_X\!=\!\|x_1+x_2\|_X,$$ 
\normalsize
and this readily implies that $\tau(x_1)$ and $\tau(x_2)$ belong to the 
same segment if and only if $x_1$ and $x_2$ do.
\end{remark}

\begin{definition}
For $x\in S_X$, we say that $x$ is flat if there is an affine hyperplane 
$H\subset X$ such that $H\cap S_X$ is a relative neighbourhood of $x$. 
\end{definition}

\begin{obs}\label{obsflat}
Let $x\in S_X$. Then, $x$ is flat if and only if $\el2(x)=\el2(x')$ for every 
$x'\in S_X$ in a relative neighbourhood of $x$. As a consequence, being flat 
is an intrinsic metric property for points in $S_X$. 
\end{obs}

\begin{proof}
By Remark~\ref{el2star}, if $H\cap S_X$ contains a (nonempty) relative open subset, 
say $U\cap S_X=\interior_{S_X}(H\cap S_X)$, then $\el2(x)=-H\cap S_X,$ for every 
$x\in U$. In particular, $\el2(x)=\el2(x')$ for $x'\in U$. 

On the other hand, let $x\in S_X$ and suppose $\el2(x)=\el2(x')$ for every 
$x'\in U$, where $U=S_X\cap (x+\e B_X)$. 
By Remark~\ref{el2star}, $\el2(x)=\el2(x')$ if and only if 
$$\{\cl{x}\in S_X:[x, \cl{x}]\subset S_X\}=\{\cl{x}\in S_X:[x', \cl{x}]\subset S_X\}.$$
Then, $[x', \cl{x}]\subset S_X$ for every 
$x', \cl{x}\in U$. This implies that $U$ is convex, so it is contained in 
$S_X\cap H$ for some hyperplane $H$. This means that $S_X\cap H$ is also a 
relative neighbourhood of $x$. 
\end{proof}

The following result has been recently published as Theorem 2.6 in \cite{wanghuang}. 
We include it here because we think that our proof is interesting enough. 

\begin{prop}[Wang, Huang,~{\cite[Theorem 2.6]{wanghuang}}]\label{segmento}
Let $\XnX$ be a two-dimensional normed space whose unit sphere contains a segment 
with length at least 1. Then, $\XnX$ has the \MUP. 
\end{prop}

\begin{proof}
Suppose that $\XnX$ is such a space and suppose there are some $\YnY$ and $\tauSXY$ 
such that $\tau$ is an onto isometry. Let $[x_1,x_2]$ be a maximal segment in $S_X$ 
such that $\|x_1-x_2\|_X\geq 1$.
Since $x_1,x_2$ are not flat points, Observation~\ref{obsflat} implies that 
neither $y_1=\tau(x_1),$ $ y_2=\tau(x_2)$ are, but Remark~\ref{x-x} implies that 
$[y_1,y_2]\subset S_Y$ and we deduce that $[y_1,y_2]$ is a maximal segment in $S_Y$. 
As $\|y_1-y_2\|_Y=\|x_1-x_2\|_X\geq 1$, the sphere $S_Y$ contains another segment 
with length at least 1. 

Take $\B_X=\{u_1,u_2\}$ as a basis of $X$, where $u_1$ and $u_2$ are 
$u_1=\frac 1{\|x_1-x_2\|_X}(x_1-x_2)$ and $u_2=\frac 1{2}(x_1+x_2)$. 
Consider $Y$ endowed with the analogous basis, 
$\B_Y=\{v_1,v_2\}$ given by $v_1=\frac 1{\|y_1-y_2\|_Y}(y_1-y_2)$ and 
$v_2=\frac 1{2}(y_1+y_2)$. We will make heavy use of 
coordinates, so please recall that they will refer to these bases 
for the remainder of the proof.  

Denoting $\lambda=\frac 12\|x_1-x_2\|_X=\frac 12\|y_1-y_2\|_Y\geq \frac 12$, 
we have $x_1=(\lambda,1), x_2=(-\lambda,1)\in S_X$ and $y_1=(\lambda,1), 
y_2=(-\lambda,1)\in S_Y$. Both spheres include the segment 
$[-\lambda,\lambda]\times\{1\}$ and so, its opposite 
$[-\lambda,\lambda]\times\{-1\}$. In both spaces we have $\|(1,0)\|=1$, so we 
actually have $([-\lambda,\lambda]\times\{-1,1\})\cup\{(\pm 1,0)\}\subset S$. 
As $\tau(\lambda,1)=(\lambda,1)$ and $\tau(-\lambda,1)=(-\lambda,1)$, 
Mankiewicz Theorem implies that $\tau(\alpha,1)=(\alpha,1)$ when 
$\alpha\in[-\lambda,\lambda]$ and, by Tingley's Theorem, we also have 
$\tau(-\alpha,-1)=(-\alpha,-1)$ for every $\alpha\in[-\lambda,\lambda]$. 

Summing up all these data, the convexity of the unit ball of any norm implies 
that every remaining point of each sphere lies inside some of the following 
four triangles. If $(\alpha, \beta)$ belongs to any of the spheres --and not 
to their above described subsets-- then:

\begin{itemize}
\item $\alpha, \beta\geq 0$ implies $(\alpha, \beta)\in \conv\{(\lambda,1),(2-\lambda,1),(1,0)\}$. 
\item $\alpha\geq 0, \beta\leq 0$ implies $(\alpha, \beta)\in \conv\{(\lambda,-1),(2-\lambda,-1),(1,0)\}$. 
\item $\alpha\leq 0, \beta\geq 0$ implies $(\alpha, \beta)\in \conv\{(-\lambda,1),(\lambda-2,1),(-1,0)\}$. 
\item $\alpha, \beta\leq 0$ implies $(\alpha, \beta)\in \conv\{(-\lambda,-1),(\lambda-2,-1),(-1,0)\}$. 
\end{itemize}

We may suppose $\alpha, \beta\geq 0$, being the other cases symmetric. 
As $\lambda\in[1/2,1]$, we have $(\alpha, \beta)\in \conv\{(1/2,1),(3/2,1),(1,0)\}$. 
We shall see that there are two more metric-depending parameters that determine 
both $\alpha$ and $\beta$ --and determined by $\alpha$ and $\beta$. 
Indeed, as the upmost part of both spheres consist of the segment 
$[(-\lambda,1),(\lambda,1)]$, for each point $(a,b)\in\R^2$ there is a cone 
where the distances to $(a,b)$ are just the differences between their second coordinates. 
Namely, if $(a',b')$ is such that $|(a-a')|\leq\lambda|(b-b')|$, then 
$\|(a,b)-(a',b')\|_X=\|(a,b)-(a',b')\|_Y=|b-b'|$. So, 
the distance from a given $(\alpha, \beta)\in \conv\{(1/2,1),(3/2,1),(1,0)\}$ 
to $(1/2,-1)$ is precisely $1+\beta$. 

This implies that {\em the} point $(\alpha, \beta)\in S$ at distance $1\leq d<2$ 
from $(1/2,-1)$ and distance smaller than 1 from $(1/2,1)$ is $(\alpha, d-1)$ for some 
$\alpha\geq\lambda$, so the second coordinate of $\tau(\alpha, \beta)$ is 
$\beta$. 

We may determine $\alpha$ by means of the metric, too. Indeed, fix 
$\beta\in[0,1]$ and take $1-\beta/2\leq\alpha\leq 1+\beta/2$, so that 
$(\alpha, \beta)\in \conv\{(1/2,1),(3/2,1),(1,0)\}$. Let $(\delta,-1)$ 
be the leftmost point in the intersection of $\R\times\{-1\}$ and 
$(\alpha, \beta)+(1+\beta)S$, it is straightforward that it is also the 
leftmost point in the intersection of $\R\times\{-1\}$ and the cone 
$\{(a,b)\in \R^2:|a-\alpha|\leq\lambda|b-\beta|\}$. 
It is clear that the inequality that defines the cone is an equality for 
$(\delta,-1)$, so $\delta$ fulfils $|(\delta-\alpha)/(-1-\beta)|=\lambda$ 
and it is obvious that $\delta<\alpha$, so 
$$(\delta-\alpha)/(-1-\beta)=\lambda.$$
As $(\delta,-1)$ is fixed, this means that $\nX$ determines $\alpha$. 
Namely, $\alpha=\delta+(1+\beta)\lambda$ where $\beta,\delta$ and $\lambda$ 
just depend on distances that agree for $\nX$ and $\nY$. This means that 
the only possibility is that $\tau(\alpha, \beta)$ is again $(\alpha, \beta)$.
\end{proof}

\begin{coro}\label{polygonal}
$\R^2$, endowed with any polygonal norm, has the \MUP. 
\end{coro}

\begin{proof}
This is just the simplest case of the main result in~\cite{KadetsMiguel}, but 
here we explicit a proof based on the proof of Proposition~\ref{segmento}. 

Let $\XnX$ be such that $S_X$ is a polygon and take some segment 
$H_0=[x_0,x_1]\subset S_X$ and $\lambda=\frac 12\|x_1-x_0\|_X$. Take another 
segment $H_1=[x_1,x_2]\subset S_X$, adjacent to $H_0$, and suppose that there 
exist $\YnY$ and $\tauSXY$ such that $\tau$ is an onto isometry. 
Consider on $X$ the basis 
$$\B_X=\{(x_1-x_0)/\|x_1-x_0\|_X, (x_1+x_0)/2\}.$$
As before, we have $\|(1,0)\|_X=1$, $x_1=(\lambda,1)$ and $x_0=(-\lambda,1)$. 
If we take the basis 
$$\B_Y=\{(\tau(x_1)-\tau(x_0))/\|\tau(x_1)-\tau(x_0)\|_Y, (\tau(x_1)+\tau(x_0))/2\},$$ 
then $\|(1,0)\|_Y=1$ and we may apply verbatim the argument in the previous 
proof to obtain 
$$\tau(\alpha,1)=(\alpha,1), \forall\ \alpha\in[-\lambda,\lambda]$$ 
and also $\tau(\alpha,\beta)=(\alpha,\beta)$ whenever $\alpha\in[\lambda,2\lambda]$ 
and $\beta>0$. This means that there is some $(\alpha,\beta)\in H_1$, $(\alpha,\beta)
\neq(\lambda,1)$ such that $\tau(\alpha,\beta)=(\alpha,\beta)$. By Mankiewicz Theorem, 
this implies that every point in $H_1$ is fixed. 

Of course, if we now {\em rotate} both $S_X$ and $S_Y$ by taking as bases 
$$\B'_X=\{(x_2-x_1)/\|x_2-x_1\|_X, (x_2+x_1)/2\} \ \mathrm{and}$$
$$\B'_Y=\{(\tau(x_2)-\tau(x_1))/\|\tau(x_2)-\tau(x_1)\|_Y, (\tau(x_2)+\tau(x_1))/2\}$$
then both rotations have the same expression in coordinates, so $\tau$ is still the 
identity on $H_0\cup H_1$. Applying the same reasoning to $H_2=[x_2,x_3]\subset S_X$ 
and so on, we obtain that $\tau$ is the identity on $S_X$. 
\end{proof}

\begin{definition}
When in two-dimensional spaces, and given a couple of linearly independent 
$x, x'\in S_X$, the {\em arc} that connects $x$ and $x'$ is defined as 
$$A(x,x')=\{\lambda x+\lambda'x':\lambda, \lambda'\geq 0\}\cap S_X$$
and it is the smallest 
connected subset of $S_X$ that contains both $x$ and $x'$. 
\end{definition}

\begin{theorem}\label{abiertotodo2d}
Let $\XnX, \YnY$ be two-dimensional normed spaces for which there exists an 
onto isometry $\tauSXY$. If there is some relative open $U\subset S_X$ where $\tau$ 
is linear, then $\tau$ is linear on $S_X$. 
\end{theorem}

\begin{proof}
Suppose there is an arc $H=A(x_1,x_2)\subset S_X$ such that 
$$\tau(\lambda_1x_1+\lambda_2x_2)=\lambda_1\tau(x_1)+\lambda_2\tau(x_2)$$ 
for every positive $\lambda_1, \lambda_2$ for which $\lambda_1x_1+\lambda_2x_2\in S_X$. 
We will show that $H$ is contained in another arc that fulfils the same condition. 
As $S_X$ is compact, this is enough. 

If $H$ contains some segment $[x',x_1]$, then the argument in the previous proofs 
shows that $\tau$ is linear in a relative neighbourhood of $x_1$, so we may suppose that 
$\el2(x_1)\cap (-H)=\{-x_1\}$. 
Taking, as usual, $\lambda=\frac 12\|x_2-x_1\|_X$, 
$$\B_X=\{(x_2-x_1)/\|x_2-x_1\|_X, (x_2+x_1)/2\} \ \mathrm{and}$$
$$\B_Y=\{(\tau(x_2)-\tau(x_1))/\|\tau(x_2)-\tau(x_1)\|_Y, (\tau(x_2)+\tau(x_1))/2\},$$
we may pass to coordinates to obtain that $\tau$ is the identity on the arc lying between 
$(-\lambda,1)$ and $(\lambda,1)$. Of course, $\tau$ is also the identity on the 
opposite arc and this implies, in particular, that $\tau(\lambda,-1)=(\lambda,-1)$. 

Let $\mu=\frac 1{\|(0,1)\|}$, so that $(0,\mu)\in S_X\cap S_Y$ and $\mu>1$, and 
observe that the points in $H$, along with the distances between them, determine 
the shape of $S_X$ (and $S_Y$) in a relative neighbourhood of $(1,0)$. In particular, 
for $z\in[0,2\lambda]\times[3-2\mu,2\mu-1]$, we have 
$$\|z-(-\lambda,1)\|_X=\|z-(-\lambda,1)\|_Y$$
because $z-(-\lambda,1)$ is close to the horizontal axis. 
In the same way, as $S_X$ and $S_Y$ coincide near $(0,\mu)$, we have 
$$\|z-(\lambda,-1)\|_X=\|z-(\lambda,-1)\|_Y$$
because $z-(\lambda,-1)$ is close to the vertical axis. 

So, if $z=(z_1,z_2)\in \big([0,2\lambda]\times[3-2\mu,2\mu-1]\big)\bigcap \big(S_X\setminus H\big)$, then 
$$a=\|\tau(z)-(\lambda,-1)\|_Y=\|z-(\lambda,-1)\|_X=\|z-(\lambda,-1)\|_Y\ \mathrm{and}$$
$$b=\|\tau(z)-(-\lambda,1)\|_Y=\|z-(-\lambda,1)\|_X=\|z-(-\lambda,1)\|_Y.$$
As $(\lambda,-1)$ does not lie in the interior of a segment included is $S_Y$, 
there are only two points in $((\lambda,-1)+aS_Y)\cap ((-\lambda,1)+bS_Y)$. 
Namely, one of these points is $z$ and the other one, say $z'$, lies at the 
other side of the line $\{t\cdot(\lambda,-1):t\in\R\}$. So, the only possibilities 
are $\tau(z)=z$ or $\tau(z)=z'$. But the Monotonicity Lemma implies that 
$$\|z'-(\lambda,1)\|\geq \min\{\|(\lambda,-1)-(\lambda,1)\|,\|(\lambda,-1)-(\lambda,1)\|\}.$$
So, assuming that 
$$\|z-(\lambda,1)\|_X\leq \min\{\|(\lambda,-1)-(\lambda,1)\|,\|(\lambda,-1)-(\lambda,1)\|\},$$ 
which we clearly can do, 
the previous reasonings lead to $\tau(z)=z$ and we have finished the proof. 
\end{proof}

\begin{coro}\label{nonstrcvx}
Every two-dimensional, non strictly convex, normed space has the \MUP. 
\end{coro}

\begin{proof}
Suppose $\XnX$ is a two-dimensional normed space and $[x,x']\subset S_X$, 
with $x'\neq \pm x$. From Mankiewicz Theorem we know that any onto isometry 
$\tauSXY$ is affine on $[x,x']$ and by Tingley's Theorem, $\{\tau(x),\tau(x')\}$ 
is a basis of $Y$. Taking coordinates with respect to these bases, we have 
$(\lambda,1-\lambda)\in S_X$, also $(\lambda,1-\lambda)\in S_Y$ and, moreover, 
$\tau(\lambda,1-\lambda)=(\lambda,1-\lambda)$ for every $\lambda\in[0,1]$. 
So, $\tau$ is linear on $[x,x']$ and now the result is clear from 
Theorem~\ref{abiertotodo2d}. 
\end{proof}

Before we proceed with the proof of Theorem~\ref{infinited}, we need this auxiliary result: 

\begin{prop}\label{normdetermines}
Let $\XnX$ be a strictly convex normed space, $x\in X$ and $U\subset S_X$ a relative open subset. 
There exists $V\subset X$, an open neighbourhood of $x$, such that every point in $V$ 
is determined by its distances to the points in $U$, i.e., if $y, y'\in V$ are such 
that $\|u-y\|_X=\|u-y'\|_X$ for every $u\in U$, then $y=y'$. 
\end{prop}

\begin{proof}
Suppose on the contrary that there are two sequences $(y_n), (y'_n)$ 
that converge to $x$ and such that $U\subset\Bis(y_n, y'_n)$ for every $n$. 

We may suppose $x\not\in \cl{U}$, the other case is obvious. Thus, the map 
$$u\in U\mapsto\sigma(u)=(u-x)/\|u-x\|_X\in S_X$$ 
is well-defined and continuous --even Lipschitz, actually. 
As $X$ is strictly convex, no line has more than two points in common with $S_X$, 
so $\sigma$ is nearly an injective map. Namely, for each $y\in S_X$, there are 
at most two points whose image is $y$ and, moreover, if $\sigma(u)=\sigma(u')$, 
then $\sigma$ is injective in a relative neighbourhood of $u$. 

So, $\sigma(U)$ contains a relative open subset of $S_X$, and this implies that $x$ is 
interior to the convex hull of $U\cup (2x-U)$. This implies that also $y_n, y_n'$ 
are interior to it for big $n$. 

But every bisector is symmetric with respect to the middle point of 
the segment, i.e., $z\in\Bis(y_n, y'_n)$ if and only if $y_n+y'_n-z\in\Bis(y_n, y'_n)$. 
Indeed, $\|z-y_n\|_X=\|z-y'_n\|_X$ implies $\|y_n+y'_n-z-y_n\|_X=\|y_n+y'_n-z-y'_n\|_X$, 
so the symmetry follows. Taking into account that $(y_n)\to x, (y'_n)\to x$ and 
$(y_n+y_n')/2\to x$, it is clear that, for big $n$, we will have both 
$y_n$ and $y_n'$ in the convex hull of 
$$\Bis(y_n, y'_n)\cup((y_n+y'_n)-\Bis(y_n, y'_n))=\Bis(y_n, y'_n).$$
This means that there are $\lambda\in[0,1]$ and $u, u'\in\Bis(y_n, y'_n)$ such 
that $y_n=\lambda u+(1-\lambda)u'$. 
We may rewrite this as: 
$$\|y_n-u\|_X=\|y'_n-u\|_X, \|y_n-u'\|_X=\|y'_n-u'\|_X \mathrm{\ and\ }y_n\in[u,u'].$$
As $\nX$ is strictly convex, the points inside a segment are determined by its distances 
to the endpoints. Indeed, $\|y_n-u\|_X+\|y_n-u'\|_X=\|u-u'\|_X$ if and only if 
$y_n\in[u,u']$. So, we have $y'_n\in[u,u']$, too. Moreover, 
$\lambda=\|y_n-u\|_X/\|u-u'\|_X$, so $y'_n=y_n$ and we are done. 
\end{proof}

\begin{remark}\label{remrm}
Our {\em a priori} impression was that there would be some result in the 
literature stating something like ``The distances to a relative open subset of 
$S_X$ determine every point in $X$ whenever $\nX$ is strictly convex". 
However, we have found nothing like this and, moreover, it seems much harder than 
expected to prove anything more general than Proposition~\ref{normdetermines}. 
Actually, a result as the supposed-to-exist one would be enough for 
proving~Conjecture~\ref{conjrn}.
\end{remark}

\begin{theorem}\label{infinited}
Let $\XnX$ be a strictly convex normed space, $\nX'$ an equivalent norm defined on $X$ 
and $\tauSXX'$ an onto isometry. If the set of fixed points of $\tau$ has nonempty 
interior, then $\tau$ is the identity and $\nX'=\nX$. 
\end{theorem}

\begin{proof}
Let $F=\{x\in S_X:\tau(x)=x\}$ be the set of fixed points of $\tau$, we will denote 
its interior by $U$. As $\tau$ is continuous, $F$ is closed. 
We shall see that it is also relative open, so $F$ must be the whole sphere $S_X$. 

Suppose $e\in F$ and take $v\in U$.
As $\nX$ is strictly convex, the distances $\{\|e-u\|_X:u\in U\}$ determine
$S_{X}$ in a neighbourhood of $(v-e)/\|v-e\|_X$, say $V$, so we have
$\|w\|_X=\|w\|_X'$ whenever $w/\|w\|_X\in V$. This means, obviously, that
$V\subset S_X'$, so we have $\|w-u\|_X'=\|w-u\|_X$ for every $w$ and $u$ in
(not necessarily relative) neighbourhoods of $e$ and $v$ respectively, i.e.,
$w\in e+\e B_X, u\in v+\e B_X$. As $\nX$ and $\nX'$ are equivalent,
the equality holds for every $w\in e+\delta B_X', u\in v+\delta B_X'$ for some
$\delta>0$. In particular, if $\|w\|_X=1$ and $w\in (e+\delta B_X')\cap (e+\e B_X)$,
then $\|w-u\|_X'=\|w-u\|_X=\|\tau(w)-u\|_X'$ for every
$u\in S_X\cap (v+\e B_X)=S_X'\cap (v+\e B_X')$.
By Proposition~\ref{normdetermines}, we are done.
\end{proof}

\begin{theorem}\label{abiertotodostrcvx}
Let $\XnX$ and $\YnY$ be finite dimensional strictly convex normed spaces and 
$\tauSXY$ an onto isometry between their unit spheres. If there is a relative open 
$U\subset S_X$ where $\tau$ is linear, then $\tau$ is linear on $S_X$. 
\end{theorem}

\begin{proof}
We just need to identify $\XnX$ and $\YnY$ with the corresponding $\R^n$ and apply 
Theorem~\ref{infinited}. 
\end{proof}

\section{Final examples and remarks}

This final section includes some less general results that, however, illustrate 
to which extent our approach can work. The end of the section includes a subsection 
where we introduce the notion of {\em normed curvature}. 

\subsection{Absolute norms in the plane}

Throughout this subsection we will restrict ourselves to the case in which $\nX, \nY$ 
are absolute, normalized, norms on $\R^2$. Please observe that this implies that 
the symmetries with respect to the axes are linear isometries in both 
$(\R^2,\nX)$ and $(\R^2,\nY)$. 

\begin{prop}\label{abstanaka}
Let $\tauSXY$ be an isometry and suppose that the only isometries of $S_X$ are $\pm\Id_X, 
\pm\phi$, where $\Id_X$ is the identity and $\phi$ is the symmetry with respect 
to the horizontal axis. Then $\tau$ is linear and, furthermore, if $\tau$ is 
not the identity, then it is one of the following maps:
\begin{itemize}
\item The rotation of angle $\pi/2, \pi$ or $3\pi/2$ around the origin. 
\item The symmetry with respect to one of the following lines: 
$$\langle(1,0) \rangle, \langle(0,1) \rangle, \langle(1,1) \rangle, \langle(1,-1) \rangle.$$
\end{itemize}
\end{prop}

\begin{proof}
Let $\tau$ be such an isometry. The group of isometries of $S_Y$ is isomorphic 
to that of $S_X$, namely its isometries are $\pm\Id_Y$ and $\pm\psi$, where 
$\psi=\tau\circ\phi\circ\tau^{-1}$. As the symmetries with respect to the axes 
are also isometries of $Y$, $\psi$ must be one of these symmetries. 

In particular, the fixed points of $\psi$ are $\pm(1,0)$ or $\pm(0,1)$ and 
the fixed points of $\phi$ are $\pm(1,0)$ and this means that $\tau(1,0)$ 
is either $(1,0), (0,1), (-1,0)$ or $(0,-1)$. 

Now, we may suppose that $\XnX$ is strictly convex, the other case is an immediate
consequence of Corollary~\ref{nonstrcvx}.
Suppose $\tau(1,0)=(1,0)$ and $\tau(0,1)=(0,1)$. 
If we show that $\tau=\Id$, then we are done because we can compose any of the other 
isometries with a linear isometry that makes the composition send 
$(1,0)$ to $(1,0)$ and $(0,1)$ to $(0,1)$.

Given any $(a,b)\in S_X$, with $a, b>0$, it is easy to determine the only 
$x\in S_X,$ $x\neq(a,b)$ such that $\|(a,b)-(1,0)\|_X=\|x-(1,0)\|_X$. Namely, 
$x$ is $(a,-b)$, and the only $x'\in S_X, x'\neq(a,b)$ such that 
$\|(a,b)-(0,1)\|_X=\|x'-(0,1)\|_X$ is obviously $x'=(-a,b)$. 
Please observe that these points are uniquely determined because of the 
strict convexity of $\XnX$ and that $\|(a,b)-x\|_X=2b, \|(a,b)-x'\|_X=2a$. 
Now, if $\tau(a,b)=(c,d)$, then the Monotonicity Lemma implies that 
$c, d>0$ and the equality $\|(a,b)-(1,0)\|_X=\|x-(1,0)\|_X$ implies 
$\|(c,d)-(1,0)\|_Y=\|\tau(x)-(1,0)\|_Y$. With the same argument than before, 
we obtain $\tau(x)=(c,-d)$, but $2a=\|(a,b)-x\|_X=\|(c,d)-\tau(x)\|_Y=2c$, so 
$c=a$. By symmetry, we have also $d=b$, so $\tau=\Id$. 
\end{proof}

\begin{remark}\label{remlen}
This statement could be seen as a cheat because it restricts the conclusion 
to norms for which we already know the group of isometries of {\em its sphere.} 
Of course, when we are trying to prove that the group of isometries of every 
sphere coincides with the group of linear isometries of the space, this may 
seem unfair. 

However, it is easy to find some norms that fit in Proposition~\ref{abstanaka}. Namely, 
if the points $\pm(1,0)$, or $\pm(0,1)$, are unique in $S_X$ in any intrinsic, 
metric, sense, then the sphere will have just the above referred isometries. 
Think, for example, in the norm whose unit sphere is a lens, see Figure~\ref{diblen}. 
\begin{figure}
\includegraphics[width=0.3\paperwidth]{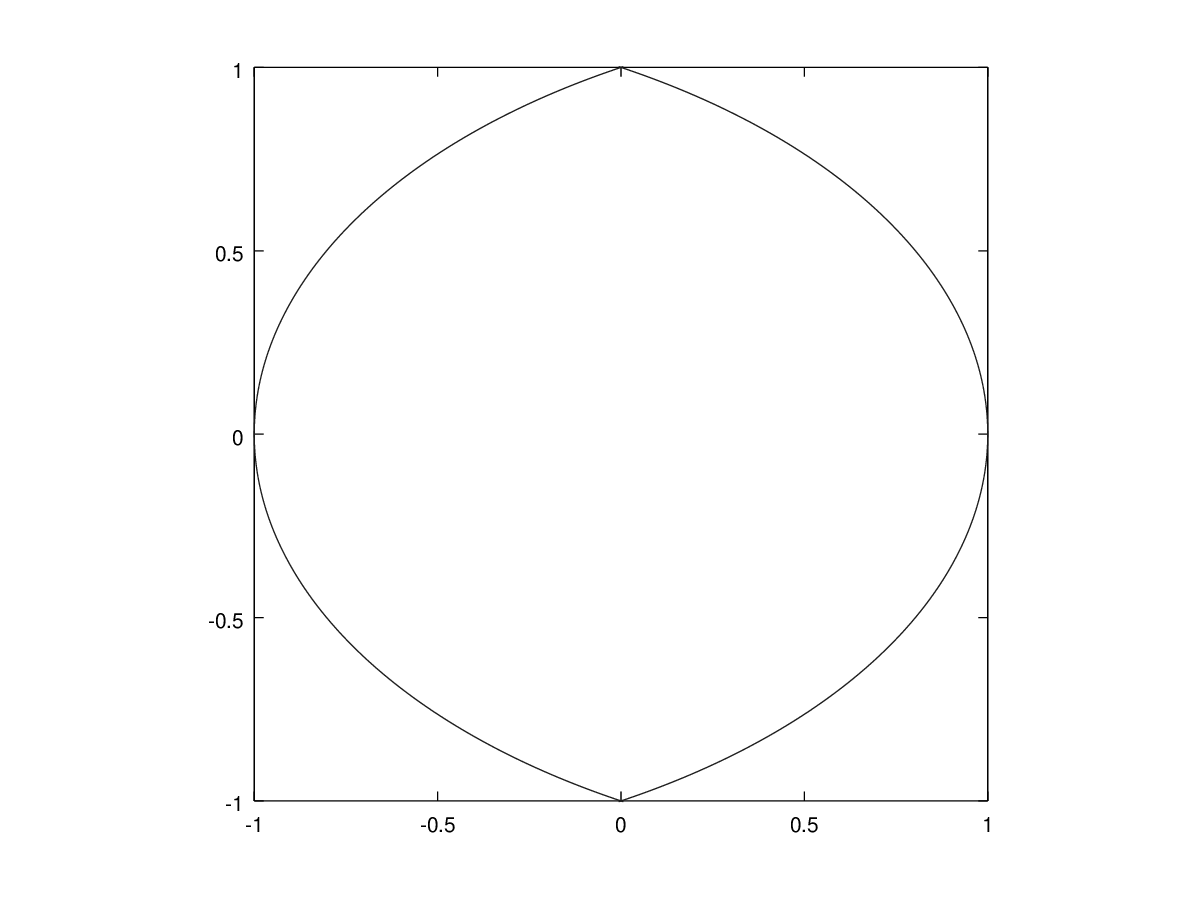}
\includegraphics[width=0.3\paperwidth]{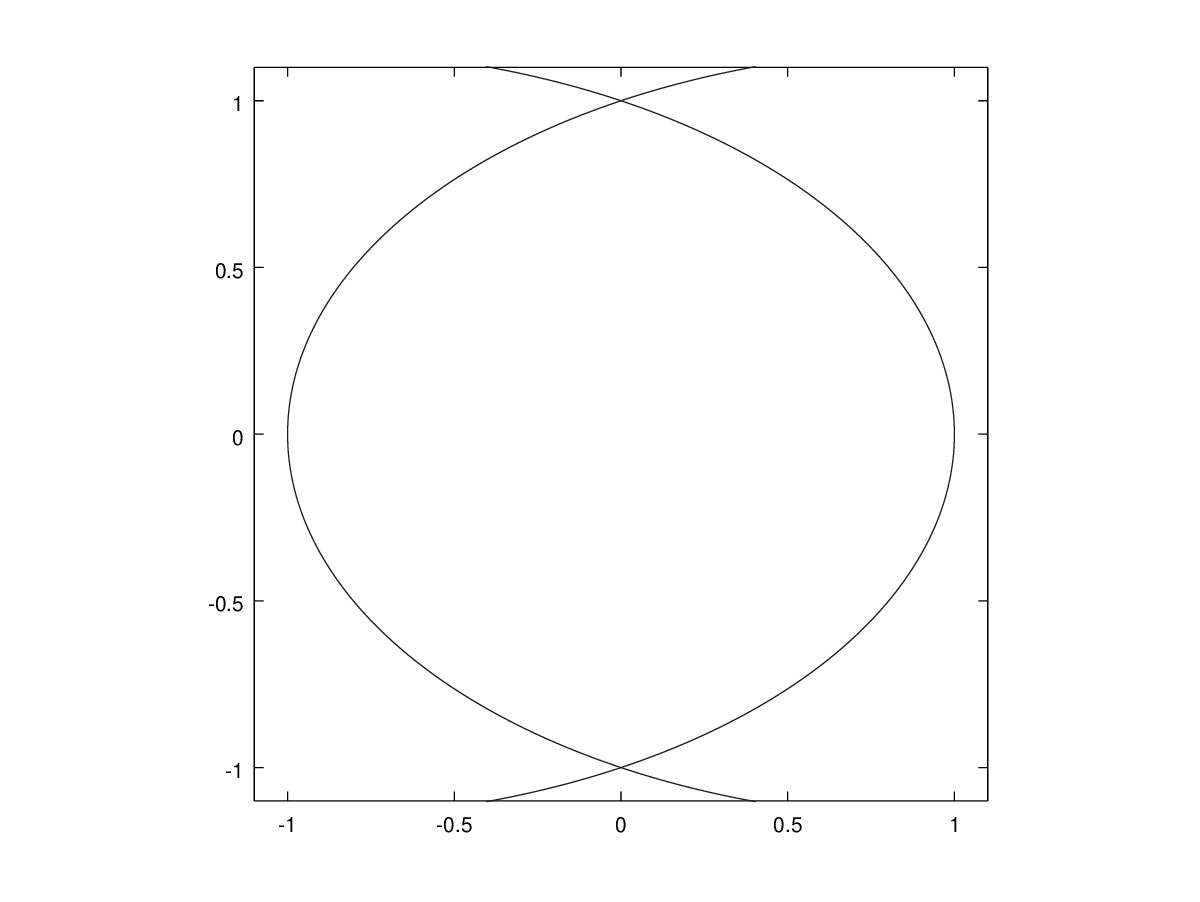}
\caption{The unit sphere in Remark~\ref{remlen}, whose associated ball is the  
intersection of two ellipses. }\label{diblen}
\end{figure}
If we consider 
the point $(0,1)$, it is clear that it is unique in some sense, namely 
$(0,1)$ and $(0,-1)$ are the only points where $S_X$ is not a 
differentiable curve. Being a point of differentiability is not, 
to the best of our knowledge, something that can be said in terms of distances 
between points of $S_X$, but there is something similar than we can say. 

Namely, if we take for each point $x\in S_X$ and every $2>\delta>0$ the only 
points $a_x(\delta),a_x'(\delta)$ such that 
$$\|x-a_x(\delta)\|_X=\|x-a'_x(\delta)\|_X,$$
it is intuitively evident that the distances $\|a_x(\delta)-a_x'(\delta)\|_X$ are 
smaller when $x=(0,\pm 1)$ than when $x$ is any other point in $S_X$. We shall 
not prove this, but the impression is that, when $\delta\to 0$, 
$\|a_x(\delta)-a_x'(\delta)\|_X/\delta$ tends to 2 if $x\neq (0,\pm 1)$ 
and that this limit is smaller than 2 for $x=(0,\pm 1)$. 
And this {\em is} measured just by means of the distances between points in the sphere, 
this is closely related to Subsection~\ref{sbscurv}.  
\end{remark}

Now, we deal with some quite more usual norms. Let us say, as in~\cite{tanakaR2}, 
that a norm $\nX$ on $\R^2$ is symmetric when $\|(x_1,x_2)\|_X=\|(x_2,x_1)\|_X$ 
for every $x=(x_1,x_2)\in\R^2$. In the next result, we consider two norms 
$\nX, \nY$ in $\R^2$ with the following characteristics: 

\begin{enumerate}
\item They are normalized. 
\item They are absolute. 
\item They are symmetric. 
\item The only isometries $\phiSXX$ are the necessary for (2) and (3), i.e., 
the rotations of angle $0, \pi/2, \pi$ and $3\pi/2$ and 
the symmetries with respect to the axis and the diagonals. 
\end{enumerate}

\begin{coro}\label{symtanaka}
With the above hypothesis, every isometry $\tauSXY$ is linear and, moreover, it 
is either one of the isometries listed in $(4)$ or one of them composed with 
the rotation of angle $\pi/4$. 
\end{coro}

\begin{proof}
This is very similar to~Proposition~\ref{abstanaka}.
\end{proof}

\begin{remark}
This Corollary may seem superfluous, but it has some interesting consequences. 
Namely, this result applies for norms having $(1,0)$ and $(0,1)$ as interchangeable 
special points, like any $p$-norm in $\R^2$. If $p>2$ and 
$\nX=\norma_p$, then $(1,0)$ and $(0,1)$ are the isosceles orthogonal points in 
$S_X$ with least distance between them. So, every isometry $\phiSXX$ must send 
$(1,0)$ to $\pm(1,0)$ or $\pm(0,1)$. The opposite happens when $p<2$: 
$(1,0)$ and $(0,1)$ have the greatest distance between isosceles orthogonal points in 
the sphere. This means that every $\norma_p$ fulfils the hypotheses of the Corollary. 
\end{remark}

\subsection{A three-dimensional example}

Here we present an example of how our approach can be meaningful also in 
three-dimensional spaces. 

\begin{lemma}
Let $\norma_{hex}$ be the norm on $\R^2$ whose unit sphere is the hexagon with 
vertices $\pm(1,0),\left(\pm \frac 12, \pm 1\right)$, i.e., 
$\|(a,b)\|_{hex}=\max\{|b|, |a|+|b|/2\}$. Then, 
$(\R^2, \norma_{hex})$ has the \MUP.
\end{lemma}

\begin{proof}
This norm has six length-1 segments in it sphere, so this lemma follows 
immediately from Proposition~\ref{segmento}.
\end{proof}

Let us recall the definition of modulus of convexity, see~\cite{day}, p. 328: 

\begin{definition}
A normed linear space $B$ is called {\em uniformly convex} if for each $\e, 
0<\e\leq 2$, there is a $\delta(\e)>0$ such that $\|b_1+b_2\|\leq 2(1-\delta(\e))$ 
if $\|b_1-b_2\|\geq\e$ and $\|b_1\|=\|b_2\|=1$; the function $\delta$ is called 
the {\em modulus of convexity of} $B$. 
\end{definition}

\noindent and the following result, Theorem 4.1 in the same outstanding work, where 
$\delta_2$ is the modulus of convexity of the Euclidean space, please observe that 
$\delta_2$ does not depend on the dimension of the space:

\begin{theorem}[Day,~{\cite[Theorem 4.1]{day}}]\label{ThDayDay}
$B$ is uniformly convex with a modulus of convexity satisfying the inequality 
$\delta(\e)\geq\delta_2(\e)$ for $0<\e\leq 2$ if and only if $B$ is an inner-product 
space and $\delta$ is identically equal to $\delta_2$. 
\end{theorem}

Since the modulus of convexity is defined just by means of some distances between 
points in the unit sphere, the characterisation of the Euclidean norm given 
by Day gives us: 

\begin{coro}\label{ThDay}
Every $\R^n$ endowed with the Euclidean norm has the Mazur-Ulam Property.
\end{coro}

Corollary~\ref{ThDay} is not new, it can be found both in~\cite{MoriOzawa19} and \cite{BCFP18}. 

This Lemma will come in handy for the proof of Example~\ref{bolibic}.

\begin{lemma}\label{BisBisBis}
Let $y_1,y_2,y_3\in S_Y$ be such that $y_i=\Bis(y_j)\cap \Bis(y_k)$ whenever 
$\{i,j,k\}=\{1,2,3\}$. Then, they are linearly independent. 
\end{lemma}

\begin{proof}
Suppose they are linearly dependent and consider the plane $H=\langle y_1,y_2,y_3\rangle$. 
The restriction of $\nY$ to $H$ is again a norm, so we have a two-dimensional 
normed space with three mutually isosceles orthogonal points in its unit sphere. 
This cannot happen because of \cite{jiliwu2011}, Corollary~2.4, so we are done. 
\end{proof}

\begin{example}\label{bolibic}
Let $\norma_X$ be the norm on $\R^3$ whose unit sphere is the revolution around the $x$-axis 
of the hexagon with vertices $\pm(1,0),\left(\pm \frac 12, \pm 1\right)$. Then, $(\R^3, \nX)$ has the \MUP. 
\end{example}

\begin{proof}[Proof of the Example.]
We will not use the explicit form of $\nX$, but it is 
$$\|(a,b,c)\|_X=\max\{\|(b,c)\|_2,|a|+\|(b,c)\|_2/2\}.$$
Suppose that there exist $\YnY$ and an onto isometry $\tauSXY$. 
Consider the vectors in the usual basis of $\R^3$, $e_1, e_2, e_3\in S_X$ 
and their images $y_1=\tau(e_1), y_2=\tau(e_2), y_3=\tau(e_3)$. 

It is clear that $e_3\in\Bis(e_1)\cap\Bis(e_2)$ and $e_2\in\Bis(e_1)$, so 
Lemma~\ref{BisBisBis} implies that $y_1,y_2,y_3$ are linearly independent. We 
may consider $Y$ endowed with the basis $\B_Y=\{y_1,y_2,y_3\}$ and we have, 
in coordinates, 
$$\tau(1,0,0)=(1,0,0),\ \tau(0,1,0)=(0,1,0),\ \tau(0,0,1)=(0,0,1).$$ 

Consider the symmetric bisector $\Bis(e_1)=\{(a_1,a_2,a_3)\in S_X:a_1=0\}$, it is 
isometric to the Euclidean unit sphere of $\R^2$. This also happens with 
$$C_X=\{x\in S_X:\|x-(1,0,0)\|_X=1\}=S_X\cap((1,0,0)+S_X)$$ 
and $C_Y=\tau(C_X)$, and the latest agrees with $S_Y\cap((1,0,0)+S_Y)$. 

Take $x\in C_X$ and $\tau(x)=(a,b,c)\in C_Y$. As its distance to $(1,0,0)$ is 1, 
we have $(a-1,b,c)\in S_Y$. We also have $\|(a-1,b,c)-(-1,0,0)\|_Y=1$, and this 
means that $(a-1,b,c)\in \tau(-C_X)=-C_Y$. Now, as the only point in $-C_X$ at 
distance 1 from $x$ is $x-(1,0,0)$, it must be $\tau(x-(1,0,0))=(a-1,b,c)\in -C_Y$. 
On the one hand, this means that the sphere $S_Y$ includes the segment $[(a,b,c),(a-1,b,c)]$. 
On the other hand, $S_Y$ must include also every other segment of the (planar) 
hexagon with vertices 
$$\{(1,0,0),(a,b,c),(a-1,b,c),(-1,0,0),(-a,-b,-c),(1-a,-b,-c),(1,0,0)\}$$
and this hexagon is the image of the hexagon whose vertices are
$$\{(1,0,0),x,x-(1,0,0),(-1,0,0),-x,(1,0,0)-x,(1,0,0)\}.$$
In particular this implies that the restriction of $\tau$ to this last hexagon 
is linear, so for this restriction to be the identity it just need to have some 
fixed point $x'\neq(\pm 1,0,0)$. This happens with both $(0,1,0)$ and $(0,0,1)$, 
so every point in, say, the horizontal and vertical hexagons is fixed. This implies 
that $\|(a,b,0)\|_Y=\|(a,b,0)\|_X$ and $\|(a,0,c)\|_Y=\|(a,0,c)\|_X$ 
for $a, b, c\in \R$, and so 
$$\|(a,b,c)-(a',b',c)\|_Y=\|(a,b,c)-(a',b',c)\|_X \ \mathrm{and\ }$$ 
$$\|(a,b,c)-(a',b,c')\|_Y=\|(a,b,c)-(a',b,c')\|_X$$ for every 
$a, b, c, a', b', c'\in \R$. Actually, the restrictions of $\nX$ and $\nY$ to 
the vertical and horizontal planes agree with $\|\cdot\|_{hex}$. 

\begin{claim} 
Let $z=(0,z_2,z_3), z'=(0,z_2',z_3')\in \Bis((1,0,0))\cap S_Y$. Then 
$\frac{z-z'}{\|z-z'\|_Y}$ also belongs to $\Bis((1,0,0))\cap S_Y$. 
\end{claim}

{\em Proof of the claim:} Let $z''=(0,z''_2,z_3'')=\frac{z-z'}{\|z-z'\|_Y}$. 
It belongs to the bisector $\Bis((1,0,0))$ if and only if 
it does not belong to any segment that contains $\pm(1,0,0)$ and 
it is the midpoint of its segment, and this is what we will see. 

Let $t,t'\in [-1/2,1/2]$. We have 
$$\|(t,z_2,z_3)-(t',z_2',z_3')\|_Y=\|(t',z_2,z_3)-(t,z_2',z_3')\|_Y,$$ 
and this implies that $\|\cdot\|_Y$ is symmetric in the segment 
$$[(-1,z_2-z_2',z_3-z_3'),(1,z_2-z_2',z_3-z_3')],\text{ i.e.}, $$
$$\|(t,z_2-z_2',z_3-z_3')\|_Y=\|(-t,z_2-z_2',z_3-z_3')\|_Y\text{ for } t\in [0,1].$$ 
In particular, as $\lambda=\|(0,z_2-z_2',z_3-z_3')\|_Y\leq 2$, we have 
$$\|(1/2,(z_2-z_2')/\lambda,(z_3-z_3')/\lambda)\|_Y=
\|(-1/2,(z_2-z_2')/\lambda,(z_3-z_3')/\lambda)\|_Y.$$ 
This readily implies that the claim holds. 

\vspace{0.2cm}

Let $V$ be the set of points in $S_X\cap\Bis(e_1)$ whose images' first coordinates vanish. 
The Claim implies that for every couple of points 
$x_1, x_2\in V$ there exists another point between them that belongs to $V$ 
too --namely, as $V$ is symmetric, $-x_2\in V$ and the inverse image of 
$(\tau(x_1)-\tau(-x_2))/\|\tau(x_1)+\tau(x_2)\|_Y$ also belongs to $V$. 
As it is obvious that $V$ is closed, we have $V=S_X\cap\Bis(e_1)$, and so
$S_Y\cap\Bis(y_1)=S_Y\cap \{(0,b,c): b, c\in\R\}$. So, $S_Y\cap\Bis(y_1)$ is 
the intersection of a plane and the unit sphere and it is isometric to the 
two-dimensional Euclidean sphere. This, along with Corollary~\ref{ThDay} 
implies that it {\em is} the two-dimensional Euclidean sphere. As 
$\tau(0,1,0)=(0,1,0)$ and $\tau(0,0,1)=(0,0,1)$, we conclude that $\tau$ is 
the identity on $\Bis(e_1)$ and this implies that it is the identity on $S_X$, 
so it is the restriction of a linear isometry. 
\end{proof}

\begin{remark}
Please observe that, whenever $\XnX$ is not strictly convex, 
there is some nonlinear isometric embedding $\R\to X$ (see~\cite{BakerIso}). 
If we still consider $\R^3$ endowed with the norm defined in Example~\ref{bolibic}, 
we have a nonlinear isometric embedding $i:(\R^2,\norma_2)\to(\R^3,\nX)$, where
$$i(x_1,x_2)=(f(x_1,x_2),x_1,x_2)\text{ and }f(x_1,x_2)=1/2\|(x_1,x_2)-(x_1',x_2')\|_2.$$ 

Actually, any $f:(\R^2,\norma_2)\to\R$ such that $f(0,0)=0$ and 
$$f(x_1,x_2)-f(x_1',x_2')\leq 1/2\|(x_1,x_2)-(x_1',x_2')\|_2$$
for every $(x_1,x_2),(x_1',x_2')\in\R^2$ would have done the trick, this example 
is just a slight modification of the last one in the same paper~\cite{BakerIso}. 

On the other hand, in the same work it is proved that every isometric embedding 
$\YnY\to\XnX$ is affine when $X$ is strictly convex. It seems that strict convexity 
can be key for the existence of nonlinear embeddings also in the spheres setting, 
so Baker's results led us to consider the following: 
\end{remark}

\begin{conjecture}
$X$ is strictly convex if and only if every isometric embedding $S_Y\to S_X$ 
extends linearly. 
\end{conjecture}

\subsection{A normed curvature}\label{sbscurv}
Before we end this paper, we need to point out a minor result that may lead to 
some interesting questions. Throughout this work, we have tried to show that 
a very important thing to have in mind when dealing with Tingley's Problem is 
the concept of {\em intrinsic metric property}. We have been able to convert 
Day's Theorem~\ref{ThDayDay} into a Mazur-Ulam Property statement just by taking into 
account that the Theorem was stated by means of distances between points in the 
unit sphere. Here we present a kind of generalisation of the curvature of a planar 
curve to curves in normed spaces, somehow in the spirit of Clarkson's modulus of Convexity: 

\begin{definition}\label{Def}
Let $\XnX$ be a normed space, $\gamma:[0,1]\to X$ a curve and $x=\gamma(t)\in X$ 
for some $0<t<1$. Suppose that there is some $c>0$ such that $\gamma\cap (x+\delta S_X)$ 
contains exactly two points for every $0<\delta<c$. 
We define the curvature of $\gamma$ at $x$ measured with $\nX$ as the following limit, 
whenever it exists: 
\begin{equation}\label{Eq}
\mathcal{K}^\gamma_{\nX}(x)=
\sqrt{\lim_{\delta\to 0}\frac{\delta-\|a-a'\|_X/2}{(\delta/2)^3}}= 
2\sqrt{\lim_{a,a'\to x}\frac{2\|x-a\|_X-\|a-a'\|_X}{(\|x-a\|_X)^3}},
\end{equation}
where $a\neq a'$ are the only points in $\gamma$ such that 
$\|x-a\|_X=\|x-a'\|_X=\delta$. 
\end{definition}

\begin{remark}
If the space $X$ is two-dimensional and the curve is its unit sphere, then the $c$ in 
Definition~\ref{Def} can always be chosen to be any $0<c<2$. 

Besides, this curvature features some desirable characteristics: 
\begin{itemize}
\item It is defined just by means of the norm and the curve, so it is generalizable to 
curves in arbitrary normed spaces. 
\item It is defined locally, or even infinitesimally. 
\item It is isometrically invariant. 
\item It is positively antihomogeneous with respect to the norm, i.e., 
$$\mathcal{K}^\gamma_{\lambda\nX}(x)=\frac 1\lambda\mathcal{K}^\gamma_{\nX}(x)$$ 
for every $\lambda>0$. 
\end{itemize}

\end{remark}

$\K$ has also this important feature: 

\begin{theorem}\label{Thm}
This definition includes the notion of the curvature of a sphere in 
$(\R^2,\|\cdot\|_2)$, i.e., $\K^{x+\lambda S_2}_{\norma_2}(x')=1/\lambda$ 
for every $x\in \R^2,\lambda>0$ and $x'\in x+\lambda S_2$.
\end{theorem}

\begin{proof}
The first we are going to show is that the curvature of the two-dimensional 
Euclidean sphere at every point measured with the Euclidean norm is 1. Let 
$x\in S_2$. 

It is clear from (\ref{Eq}) that our definition of curvature only depends on 
the distances, so every isometry must preserve the curvature. In particular, as 
$S_2$ is isometrically homogeneous, this means that its curvature is constant so 
we may suppose $x=(1,0)$. 

The points near $(1,0)$ are $(\cos(t),\sin(t))$ for $t\in[-\e,\e]$, and the 
points $a, a'$ in~(\ref{Eq}) are $a=(\cos(t),\sin(t)), a'=(\cos(t),-\sin(t))$ 
for some positive $t$. 
We need to evaluate the right hand expression in~(\ref{Eq}). We have 

\begin{eqnarray*}
\begin{aligned}
& \mathcal{K}^{S_2}_{\norma_2}((1,0))= 
2\sqrt{\lim_{a,a'\to x}\frac{2\|x-a\|_X-\|a-a'\|_X}{(\|x-a\|_X)^3}} = 
\\
&  =2\sqrt{\lim_{t\to 0}\frac
{2\|(1,0)-(\cos(t),\sin(t))\|_2-\|(\cos(t),\sin(t))-(\cos(t),-\sin(t))\|_2}
{(\|(1,0)-(\cos(t),\sin(t))\|_2)^3}}= 
\\
& =2\sqrt{\lim_{t\to 0}\frac{2\sqrt{2-2\cos(t)}-2\sin(t)}{\sqrt{(2-2\cos(t))^3}}}= 
\sqrt{8}\sqrt{\lim_{t\to 0}\frac{\sqrt{2-2\cos(t)}-\sin(t)}{\sqrt{(2-2\cos(t))^3}}}. 
\end{aligned}
\end{eqnarray*}

For the sake of clarity, we are going to leave the limit as simple as we can. 
What we will actually compute is 
\begin{eqnarray}\label{limite}
\lim_{t\to 0}\frac{\sqrt{2-2\cos(t)}-\sin(t)}{\sqrt{(2-2\cos(t))^3}},
\end{eqnarray}
we need to show that it equals $1/8$. In the remaining of the proof we will 
avoid the $t\to 0$ and will not recall that $t>0$. As we will make heavy use of 
L'{}Hôpital's Rule, we will write as $\stackrel{*}=$ the equalities given by this Rule. 

We will need this later: 
\begin{eqnarray}\label{limitefacil}
\lim_{t\to 0}\frac{\sin(t)}{\sq22}=1.
\end{eqnarray}
This limit is 1 if and only if 
$$1=\lim\frac{\sin^2(t)}{2-2\cos(t)}\LH \lim\frac{2\sin(t)\cos(t)}{2\sin(t)}=\lim \cos(t),$$
so~(\ref{limitefacil}) holds. 

It is clear that in~(\ref{limite}) both the numerator and the denominator tend to 0. 
In the following there is one case where some explanation is needed, but we will 
leave the explanation for the end of the proof. 
\begin{eqnarray*}
\begin{aligned}
& 
\lim\frac{\sqrt{2-2\cos(t)}-\sin(t)}{\sqrt{(2-2\cos(t))^3}}\LH
\lim\frac{\sin(t)/\sq22-\cos(t)}{3\sin(t)\sq22}=
\\ & 
=\lim\frac{\sin(t)-\cos(t)\sq22}{3\sin(t)(2-2\cos(t))}\LH
\end{aligned}
\end{eqnarray*}
\begin{eqnarray*}
\begin{aligned}
&
=\frac 13\lim\frac{\cos(t)-\sin(t)\cos(t)/\sq22+\sin(t)\sq22}{2\sin^2(t)\cos(t)(2-2\cos(t))}\LH 
\\ &
=\frac 13\lim\frac{-\sin(t)+\sq22 \cos(t) +\frac{2\sin^2(t)-\cos(t)}{\sq22} +\frac{\cos(t)\sin^2(t)}{\sq22^3}}
{2\sin2(t)(4\cos(t)-1)}\LH
\\ &
=\frac 13\lim\frac{-\cos(t)+\frac{\sin(t)\cos(t)}{\sq22}
\left[2+\frac{\cos(t)}{2-2\cos(t)}-\frac{3\sin^2(t)}{(2-2\cos(t))^2}+1+\frac{2\cos(t)}{2-2\cos(t)}+4 \right]}
{-8\sin^2(t)+8\cos^2(t)-2\cos(t)}+
\\ &
+\frac 13 \lim\frac{-3\left(\frac{\sin(t)}{\sq22}\right)^3 -\sin(t)\sq22}
{-8\sin^2(t)+8\cos^2(t)-2\cos(t)}=
\\ &
=\frac 13\lim\frac {-1+\frac{\sin(t)\cos(t)}{\sq22}\cdot\left[7 +\frac{3\cos(t)}{2-2\cos(t)}-\frac{3\sin^2(t)}{(2-2\cos(t))^2} \right]
-3\left(\frac{\sin(t)}{\sq22}\right)^3 }6=
\\ &
=\lim\frac { \frac{\sin(t)\cos(t)}{\sq22}\left(\frac{\cos(t)}{2-2\cos(t)}-\frac{\sin^2(t)}{(2-2\cos(t))^2}\right) +1}{6}.
\end{aligned}
\end{eqnarray*}
\vspace{2mm}

\noindent
So, what we need right now is to show that 
\begin{eqnarray*}
\begin{aligned}
-\frac 14&=\lim\frac{\sin(t)\cos(t)}{\sq22}
\left(\frac{\cos(t)}{2-2\cos(t)}-\frac{\sin^2(t)}{(2-2\cos(t))^2}\right)=
\\ &
=\lim\frac{\cos(t)(2-2\cos(t))-\sin^2(t)}{(2-2\cos(t))^2}=
\lim\frac{2\cos(t)-2\cos^2(t)-\sin^2(t)}{4+4\cos^2(t)-8\cos(t)}=
\\ &
=\lim\frac{2\cos(t)-\cos^2(t)-1}{4+4\cos^2(t)-8\cos(t)}=
\lim\frac{-(1-2\cos(t)+\cos^2(t))}{4(1-2\cos(t)+\cos^2(t))},
\end{aligned}
\end{eqnarray*}

\vspace{2mm}
\noindent so we have proved that~(\ref{limite}) holds. The last application of 
L'{}Hôpital's Rule is the only tricky one, here we show that the numerator 
tends to 0. The only problem could be here: 

\begin{eqnarray*}
\lim\left(\frac{2\sin^2(t)-\cos(t)}{\sq22} +\frac{\cos(t)\sin^2(t)}{\sq22^3}\right)
\end{eqnarray*}

As we have~(\ref{limitefacil}), this reduces to 

\begin{eqnarray*}
\lim\!\left(\frac{-\cos(t)}{\sq22} +\frac{\cos(t)\sin^2(t)}{\sq22^3}\right)\!=
\lim\!\left(\!\cos(t)\frac{-(2-2\cos(t))+\sin^2(t)}{\sq22^3}\right)\LH
\end{eqnarray*}
\begin{eqnarray*}
=\lim\frac{2\sin(t)\cos(t)-2\sin(t)}{3\sin(t)\sq22}
=\lim\frac{2\cos(t)-2}{3\sq22}\LH
\lim\frac{-2\sin(t)}{3\sin(t)/\sq22} 
\end{eqnarray*}
and the last limit is just 
$$-\frac 23\sq22\to 0,$$
so the first part of the proof is finished. 

The usual curvature of $\lambda S_2$ is $1/\lambda$ at every point for each 
$\lambda>0$ and, as the formula given for $\K^\gamma_{\nX}$ is clearly positively 
antihomogeneous, it is clear that 
$$\K^{\lambda S_2}_{\norma_2}(\lambda x)=\frac 1\lambda$$
for every $x\in S_2$, so both curvatures agree at every point of any centred 
sphere. Both are translation invariant, so they agree at every sphere and 
we have finished the proof. 
\end{proof}

\section{Aknowledgments}

Supported in part by DGICYT project MTM2016$\cdot$76958$\cdot$C2$\cdot$1$\cdot$P 
(Spain) and Junta de Extremadura programs GR$\cdot$15152 and  IB$\cdot$16056. 

\section*{References}

\bibliography{elsReflectionTingley}

\end{document}